\documentclass[reqno]{amsart}
\usepackage{amsmath, amsthm, amssymb, color}
\usepackage{graphicx}
\usepackage[mathscr]{euscript}
\usepackage{latexsym}
\usepackage{hyperref}
\usepackage{xcolor}
\usepackage{amsmath,amssymb,subeqnarray,amsfonts,graphics,epsfig,comment,color}
\usepackage{subfigure,txfonts}
\usepackage{latexsym}
 \usepackage{mathptmx}      
\usepackage{color}
\usepackage[numbers,sort&compress]{natbib}
\numberwithin{equation}{section}
\newtheorem{Theorem}{Theorem}[section]
\newtheorem{Lemma}[Theorem]{Lemma}

\theoremstyle{definition}

\newtheorem{remark}[Theorem]{Remark}
\newtheorem{Proposition}[Theorem]{Proposition}

\newcounter{RomanNumber}

\def\be{\begin{equation}}
\def\en{\end{equation}}
\def\bs{\begin{split}}
\def\es{\end{split}}

\allowdisplaybreaks
\title[On instability  of a generic two--fluid model] {On instability of a generic compressible
two--fluid model in $\mathbb R^3$}
\author{Guochun Wu}
\address{Guochun Wu \newline Fujian Province University Key Laboratory of Computational Science, School of Mathematical Sciences, Huaqiao University, Quanzhou 362021, P.R. China.}
\email{guochunwu@126.com}
\author{Lei Yao}
\address{Lei Yao \newline School of Mathematics and Statistics, Northwestern Polytechnical University, Xi'an 710129, P.R. China.
\newline School of Mathematics and Center for Nonlinear Studies, Northwest University, Xi'an 710127,P.R. China.}
\email{yaolei1056@hotmail.com}
\author{Yinghui Zhang*}
\address{Yinghui Zhang \newline School of Mathematics and Statistics, Guangxi Normal University, Guilin, Guangxi 541004, P.R.
China} \email{yinghuizhang@mailbox.gxnu.edu.cn}
\bigbreak

\thanks{* Corresponding author: yinghuizhang@mailbox.gxnu.edu.cn}
\bigbreak

\date{\today}
\usepackage{hyperref}

\begin{document}
\begin{abstract}
We are concerned with the instability of a generic compressible two--fluid
model in the whole space $\mathbb{R}^3$, where the capillary pressure $f(\alpha^-\rho^-)=P^+-P^-\neq 0$ is taken into account.
For the case that the capillary pressure is a strictly decreasing
function near the equilibrium, namely, $f'(1)<0$,
  [Evje--Wang--Wen, Arch Rational Mech Anal
221:1285--1316, 2016] established global stability of the constant equilibrium state for the three--dimensional
Cauchy problem under some
smallness assumptions. Recently, [Wu--Yao--Zhang, arXiv:2204.10706] proved global stability of the constant equilibrium state for the case $P^+=P^-$ (corresponding to $f'(1)=0$).
In this work, we investigate the instability of the constant equilibrium state for the case that the capillary pressure is a strictly increasing function near the equilibrium, namely, $f'(1)>0$.
First, by employing Hodge decomposition technique and making detailed analysis of the Green's function for the corresponding linearized system, we construct solutions
of the linearized problem that grow exponentially in time in the Sobolev space $H^k$, thus leading to a global
instability result for the linearized problem.
Moreover, with the help of the global
linear instability result and a local existence theorem of classical solutions to the original nonlinear system, we can then
show the instability of the nonlinear problem in the sense of Hadamard by making a delicate analysis on the properties of the semigroup.
Therefore, our result shows that for the case $f'(1)>0$, the constant equilibrium state of the two--fluid model is linearly globally unstable and nonlinearly locally unstable
in the sense of Hadamard, which is in contrast to
the cases $f'(1)<0$ ([Evje--Wang--Wen, Arch Rational Mech Anal
221:1285--1316, 2016]) and $P^+=P^-$ (corresponding to $f'(1)=0$) ([Wu--Yao--Zhang, arXiv:2204.10706]) where the constant equilibrium state of the two--fluid model \eqref{1.5} is nonlinearly globally stable.
\end{abstract}

\maketitle

\section{\leftline {\bf{Introduction.}}}
\setcounter{equation}{0}
\subsection{Background and motivation}
As is well--known, most of the flows in nature are multi--fluid
flows. Such a terminology includes the flows of non--miscible fluids
such as air and water; gas, oil and water. For the flows of miscible
fluids, they usually form a ``new" single fluid possessing its own
rheological properties. One interesting example is the stable
emulsion between oil and water which is a non--Newtonian fluid, but
oil and water are Newtonian ones.\par
 One of the classic examples of multi--fluid flows is
small amplitude waves propagating at the interface between air and
water, which is called a separated flow. In view of modeling, each
fluid obeys its own equation and couples with each other through the
free surface in this case. Here, the motion of the fluid is governed
by the pair of compressible Euler equations with free surface:
\begin{align}
\partial_{t} \rho_{i}+\nabla \cdot\left(\rho_{i} v_{i}\right) &=0, \quad i=1,2,\label{1.1} \\
\partial_{t}\left(\rho_{i} v_{i}\right)+\nabla \cdot\left(\rho_{i} v_{i} \otimes v_{i}\right)+\nabla p_i &=-g\rho_{i}
e_3\pm F_D.\label{1.2}
\end{align}
In above equations, $\rho_1$ and $v_1$ represent the density and
velocity of the upper fluid (air), and  $\rho_2$ and $v_2$ denote
the density and velocity of the lower fluid (water). $p_{i}$ denotes
the pressure. $-g\rho_{i} e_3$ is the gravitational force with the
constant $g>0$ the acceleration of gravity and $e_3$ the vertical
unit vector, and $F_D$ is the drag force. As mentioned before, the
two fluids (air and water) are separated by the unknown free surface
$z=\eta(x, y, t)$, which is advected with the fluids according to
the kinematic relation:
\begin{equation}\partial_t\eta=v_{1,z}-v_{1,x}\partial_x \eta-v_{1, y}\partial_y \eta\label{1.3}\end{equation}
on two sides of the surface $z=\eta$ and the pressure is continuous
across this surface.\par When the wave's amplitude becomes large
enough, wave breaking may happen. Then, in the region around the
interface between air and water, small droplets of liquid appear in
the gas, and bubbles of gas also appear in the liquid. These
inclusions might be quite small. Due to the appearances of collapse
and fragmentation, the topologies of the free surface become quite
complicated and a wide range of length scales are involved.
Therefore, we encounter the situation where two--fluid models become
relevant if not inevitable. The classic approach to simplify the
complexity of multi--phase flows and satisfy the engineer's need of
some modeling tools is the well--known volume--averaging method (see
\cite{Ishii1, Prosperetti} for details). Thus, by performing such a
procedure, one can derive a model without surface: a two--fluid
model. More precisely, we denote $\alpha^{\pm}$ by the volume
fraction of the liquid (water) and gas (air), respectively.
Therefore, $\alpha^++\alpha^-=1$. Applying the volume--averaging
procedure to the equations \eqref{1.1} and \eqref{1.2} leads to the
following generic compressible two--fluid model:
\begin{equation}\label{1.4}
\left\{\begin{array}{l}
\partial_{t}\left(\alpha^{\pm} \rho^{\pm}\right)+\operatorname{div}\left(\alpha^{\pm} \rho^{\pm} u^{\pm}\right)=0, \\
\partial_{t}\left(\alpha^{\pm} \rho^{\pm} u^{\pm}\right)+\operatorname{div}\left(\alpha^{\pm} \rho^{\pm} u^{\pm} \otimes u^{\pm}\right)
+\alpha^{\pm} \nabla P^\pm=-g\alpha^{\pm}\rho^{\pm} e_3\pm F_D.
\end{array}\right.
\end{equation}
\par
We have already discussed the case of water waves, where a separated
flow can lead to a two--fluid model from the viewpoint of practical
modeling. As mentioned before, two--fluid flows are very common in
nature, but also in various industry applications such as nuclear
power, chemical processing, oil and gas manufacturing. According to
the context, the models used for simulation may be very different.
However, averaged models share the same structure as \eqref{1.4}. By
introducing viscosity effects and capillary pressure effects, one can generalize the above system
\eqref{1.4} to
\begin{equation}\label{1.5}
\left\{\begin{array}{l}
\partial_{t}\left(\alpha^{\pm} \rho^{\pm}\right)+\operatorname{div}\left(\alpha^{\pm} \rho^{\pm} u^{\pm}\right)=0, \\
\partial_{t}\left(\alpha^{\pm} \rho^{\pm} u^{\pm}\right)+\operatorname{div}\left(\alpha^{\pm} \rho^{\pm} u^{\pm} \otimes u^{\pm}\right)
+\alpha^{\pm} \nabla P^{\pm}\left(\rho^{\pm}\right)=\operatorname{div}\left(\alpha^{\pm} \tau^{\pm}\right), \\
P^{+}\left(\rho^{+}\right)-P^{-}\left(\rho^{-}\right)=f\left(\alpha^{-}
\rho^{-}\right),
\end{array}\right.
\end{equation} where
$\rho^{\pm}(x, t) \geqq 0, u^{\pm}(x, t)$ and
$P^{\pm}\left(\rho^{\pm}\right)=A^{\pm}\left(\rho^{\pm}\right)^{\bar{\gamma}^{\pm}}$
denote the densities, the velocities of each phase, and the two
pressure functions, respectively. $\bar{\gamma}^{\pm} \geqq 1,
A^{\pm}>0$ are positive constants. In what follows, we set
$A^{+}=A^{-}=1$ without loss of any generality. As in \cite{Evje9},
we assume that the capillary pressure $f$ belongs to $C^{3}([0, \infty))$. Moreover,
$\tau^{\pm}$ are the viscous stress tensors
\begin{equation}\label{1.6}
\tau^{\pm}:=\mu^{\pm}\left(\nabla u^{\pm}+\nabla^{t}
u^{\pm}\right)+\lambda^{\pm} \operatorname{div} u^{\pm} \mathrm{Id},
\end{equation}
where the constants $\mu^{\pm}$ and $\lambda^{\pm}$ are shear and
bulk viscosity coefficients satisfying the physical condition:
$\mu^{\pm}>0$ and $2 \mu^{\pm}+3 \lambda^{\pm} \geqq 0,$ which
implies that $\mu^{\pm}+\lambda^{\pm}>0.$ For more information
about this model, we refer to \cite{Brennen1, Bresch1, Bresch2,
Friis1, Ishii1, Prosperetti, Raja} and references therein.
However, it is well--known
that as far as mathematical analysis of two--fluid model is
concerned, there are many technical challenges. Some of them
involve, for example:
\begin{itemize}
\item The two--fluid model is a partially dissipative system.
More precisely, there is no dissipation on the mass conservation
equations, whereas the momentum equations have viscosity
dissipations;

\item The corresponding linear system of the model has
zero eigenvalue, which makes mathematical analysis (well--posedness
and stability) of the model become quite difficult and complicated;

\item Transition to single--phase regions, i.e, regions where the mass
$\alpha^{+} \rho^{+}$ or $\alpha^{-} \rho^{-}$ becomes zero, may
occur when the volume fractions $\alpha^{\pm}$ or the densities
$\rho^{\pm}$ become zero;

\item The system is non--conservative, since the non--conservative terms $\alpha^{\pm} \nabla
P^{\pm}$ are involved in the momentum equations. This brings various
 mathematical difficulties for us to employ methods used
for single phase models to the two--fluid model.

\end{itemize}\par
For the case that the capillary pressure is a strictly decreasing
function near the equilibrium, namely, $f'(1)<0$, Evje--Wang--Wen \cite{Evje9} obtained global stability of the constant equilibrium state for the three--dimensional
Cauchy problem of the two--fluid model \eqref{1.5} under the assumption that the initial perturbation is small in $H^2$-norm and bounded in $L^1$-norm. It should be noted that as pointed out by Evje--Wang--Wen in \cite{Evje9}, the assumption $f'(1)<0$
played a crucial role in their analysis and appeared to have an essential stabilization effect on the
model in question.
Bretsch et al. in
the seminal work \cite{Bresch1} considered a model similar to
\eqref{1.5}. More specifically, they made the following assumptions:
\begin{itemize} \item $P^{+}=P^{-}$ (particularly, $f'(1)=0$ in this case);\\
\item inclusion of viscous terms of the form \eqref{1.2} where $\mu^{\pm}$
depends on densities $\rho^{\pm}$ and $\lambda^{\pm}=0$;\\
\item inclusion of a third order derivative of $\alpha^{\pm}
\rho^{\pm}$, which are so--called  internal capillary forces
represented by the well--known Korteweg model on each phase.
\end{itemize}
They obtained the global weak solutions in the periodic domain with
$1<\overline{\gamma}^{\pm}< 6$. It is worth
mentioning that the method of \cite{Bresch1} doesn't work for the case without
the internal capillary forces. Later, Bresch--Huang--Li
\cite{Bresch2} established the global existence of weak solutions in
one space dimension without the internal capillary forces when
$\overline{\gamma}^{\pm}>1$ by taking
advantage of the one space dimension. However, the method of \cite{Bresch2} relies
crucially on the advantage of one space dimension, and particularly cannot be applied for high
dimensional problem. Recently, Wu--Yao--Zhang \cite{WYZ} showed
the global stability of the constant equilibrium state in three space dimension by
 exploiting the
dissipation structure of the model (with $P^+=P^-$ and without internal capillary forces) and making full use of several key observations.
For the case of
the special density-dependent viscosities with equal viscosity coefficients and the case of general constant
viscosities, Cui--Wang--Yao--Zhu
\cite{c1} and Li--Wang--Wu--Zhang \cite {LWWZ} proved the global stability of the constant equilibrium state
 for the three--dimensional
Cauchy problem with the internal capillary forces, respectively.
 \par
 To sum up,  the works \cite{Evje9} and \cite{WYZ} rely essentially on the assumption $f'(1)<0$ and
 $P^+=P^-$ (corresponding to $f'(1)=0$). Therefore, a natural and
important problem is that what will happen for the case that the capillary pressure is a strictly increasing
function near the equilibrium, namely, $f'(1)>0$. That is to say,
what about the stability of three--dimenional
Cauchy problem to the two--fluid model \eqref{1.5} with $f'(1)>0$. The main purpose of this work is to give a definite answer to this issue.
More precisely, we first employ Hodge decomposition technique and make detailed analysis of the Green's function for the corresponding linearized system to construct solutions
of the linearized problem that grow exponentially in time in the Sobolev space $H^k$, thus leading to a global
instability result for the linearized problem.
Then, based on the global
linear instability result and a local existence theorem of classical solutions to the original nonlinear system, we can
prove the instability of the nonlinear problem in the sense of Hadamard by making a delicate analysis on the properties of the semigroup.
Therefore, our result shows that for the case $f'(1)>0$, the constant equilibrium state of the two--fluid model \eqref{1.5} is linearly globally unstable and nonlinearly locally unstable
in the sense of Hadamard, which is in contrast to
the cases $f'(1)<0$ (\cite{Evje9}) and $P^+=P^-$ (corresponding to $f'(1)=0$) (\cite{WYZ}) where the constant equilibrium state of the two--fluid model \eqref{1.5} is nonlinearly globally stable.

\subsection{New formulation of system \eqref{1.5} and Main Results}
In this subsection, we devote ourselves to reformulating the system
\eqref{1.5} and stating the main results. To begin with, noting
the relation between the pressures of \eqref{1.5}$_3$, one has
\begin{equation}\label{1.7}
\mathrm{d} P^{+}-\mathrm{d} P^{-}=\mathrm{d} f\left(\alpha^{-}
\rho^{-}\right),
\end{equation}
where $P^{\pm}:=P^{\pm}\left(\rho^{\pm}\right).$ It is clear that
\[
\mathrm{d} P^{+}=s_{+}^{2} \mathrm{d} \rho^{+}, \quad \mathrm{d}
P^{-}=s_{-}^{2} \mathrm{d} \rho^{-}, \quad \text { where }
s_{\pm}^{2}:=\frac{\mathrm{d} P^{\pm}}{\mathrm{d}
\rho^{\pm}}\left(\rho^{\pm}\right)=\bar{\gamma}^{\pm}
\frac{P^{\pm}\left(\rho^{\pm}\right)}{\rho^{\pm}}.
\]
Here $s_{\pm}$ represent the sound speed of each phase respectively.
Motivated by \cite{Bresch1}, we introduce the fraction densities
\begin{equation}\label{1.8}
R^{\pm}=\alpha^{\pm} \rho^{\pm},
\end{equation}
which together with the fact that $\alpha^++\alpha^-=1$ gives
\begin{equation}\label{1.9}
\mathrm{d} \rho^{+}=\frac{1}{\alpha_{+}}\left(\mathrm{d}
R^{+}-\rho^{+} \mathrm{d} \alpha^{+}\right), \quad \mathrm{d}
\rho^{-}=\frac{1}{\alpha_{-}}\left(\mathrm{d} R^{-}+\rho^{-}
\mathrm{d} \alpha^{+}\right).
\end{equation}
By virtue of \eqref{1.7} and \eqref{1.9}, we finally get
\begin{equation}\label{1.10}
\mathrm{d} \alpha^{+}=\frac{\alpha^{-} s_{+}^{2}}{\alpha^{-}
\rho^{+} s_{+}^{2}+\alpha^{+} \rho^{-} s_{-}^{2}} \mathrm{d}
R^{+}-\frac{\alpha^{+} \alpha^{-}}{\alpha^{-} \rho^{+}
s_{+}^{2}+\alpha^{+} \rho^{-}
s_{-}^{2}}\left(\frac{s_{-}^{2}}{\alpha^{-}}+f^{\prime}\right)
\mathrm{d} R^{-}. \end{equation}
 Substituting \eqref{1.10} into \eqref{1.9}, we deduce
the following expressions:
\[
\mathrm{d} \rho^{+}=\frac{\rho^{+} \rho^{-}
s_{-}^{2}}{R^{-}\left(\rho^{+}\right)^{2}
s_{+}^{2}+R^{+}\left(\rho^{-}\right)^{2} s_{-}^{2}}\left(\rho^{-}
\mathrm{d} R^{+}+\left(\rho^{+}+\rho^{+} \frac{\alpha^{-}
f^{\prime}}{s_{-}^{2}}\right) \mathrm{d} R^{-}\right),
\]
and
\[
\mathrm{d} \rho^{-}=\frac{\rho^{+} \rho^{-}
s_{+}^{2}}{R^{-}\left(\rho^{+}\right)^{2}
s_{+}^{2}+R^{+}\left(\rho^{-}\right)^{2} s_{-}^{2}}\left(\rho^{-}
\mathrm{d} R^{+}+\left(\rho^{+}-\rho^{-} \frac{\alpha^{+}
f^{\prime}}{s_{+}^{2}}\right) \mathrm{d} R^{-}\right),
\]
which together with \eqref{1.7} gives the pressure differential
$\mathrm{d} P^{\pm}$
\[
\mathrm{d} P^{+}=\mathcal{C}^{2}\left(\rho^{-} \mathrm{d}
R^{+}+\left(\rho^{+}+\rho^{+} \frac{\alpha^{-}
f^{\prime}}{s_{-}^{2}}\right) \mathrm{d} R^{-}\right) ,\] and
\[
\mathrm{d} P^{-}=\mathcal{C}^{2}\left(\rho^{-} \mathrm{d}
R^{+}+\left(\rho^{+}-\rho^{-} \frac{\alpha^{+}
f^{\prime}}{s_{+}^{2}}\right) \mathrm{d} R^{-}\right) ,\] where
\[
\mathcal{C}^{2}:=\frac{s_{-}^{2} s_{+}^{2}}{\alpha^{-} \rho^{+}
s_{+}^{2}+\alpha^{+} \rho^{-} s_{-}^{2}}.\]\par \noindent Next, by noting the
fundamental relation: $\alpha^++\alpha^-=1$, we can get the
following equality:
\begin{equation}\label{1.11}
\frac{R^+}{\rho^+}+\frac{R^-}{\rho^-}=1, ~~\hbox{and thus}~~
\rho^-=\frac{R^-\rho^+}{\rho^+-R^+}.\end{equation} Then, it holds
from the pressure relation $\eqref{1.5}_3$ that
\begin{equation}\label{1.12}
\varphi(\rho^+, R^+,
R^-):=P^+(\rho^+)-P^-{\left(\frac{R^-\rho^+}{\rho^+-R^+}\right)}-f({R^-})=0.
\end{equation}

\noindent Thus, we can employ the implicit function theorem to
define $\rho^{+}$. To see this, by differentiating the above
equation with respect to $\rho^{+}$ for given $R^{+}$ and $R^{-}$,
we get
\[
\frac{\partial\varphi}{\partial\rho^+}(\rho^+, R^+,
R^-)=s_{+}^{2}+s_{-}^{2} \frac{R^{-}
R^{+}}{\left(\rho^{+}-R^{+}\right)^{2}},
\]
which is positive for any $\rho^{+}\in(R^+, +\infty)$ and
$R^{\pm}>0.$ This together with the implicit function theorem
implies that $\rho^{+}=\rho^{+}\left(R^{+}, R^{-}\right)
\in\left(R^{+},+\infty\right)$ is the unique solution of the
equation \eqref{1.12}. By virtue of \eqref{1.8}, \eqref{1.12} and
 the
fundamental fact that $\alpha^++\alpha^-=1$, $\rho^{-}$ and $\alpha^{\pm}$ can be defined by
\[
\begin{aligned}
\rho^{-}\left(R^{+}, R^{-}\right) &=\frac{R^{-} \rho^{+}\left(R^{+}, R^{-}\right)}{\rho^{+}\left(R^{+}, R^{-}\right)-R^{+}}, \\
\alpha^{+}\left(R^{+}, R^{-}\right) &=\frac{R^{+}}{\rho^{+}\left(R^{+}, R^{-}\right)}, \\
\alpha^{-}\left(R^{+}, R^{-}\right)
&=1-\frac{R^{+}}{\rho^{+}\left(R^{+},
R^{-}\right)}=\frac{R^{-}}{\rho^{-}\left(R^{+}, R^{-}\right)}.
\end{aligned}
\]
We refer the readers to [\cite{Bresch2}, P. 614] for more details.
\par
Therefore, we can rewrite system \eqref{1.5} into the following
equivalent form:
\begin{equation}\label{1.13}
\left\{\begin{array}{l}
\partial_{t} R^{\pm}+\operatorname{div}\left(R^{\pm} u^{\pm}\right)=0, \\
\partial_{t}\left(R^{+} u^{+}\right)+\operatorname{div}\left(R^{+} u^{+} \otimes u^{+}\right)+\alpha^{+} \mathcal{C}^{2}\left[\rho^{-} \nabla R^{+}+\left(\rho^{+}+\rho^{+} \frac{\alpha^{-} f^{\prime}}{s_{-}^{2}}\right) \nabla R^{-}\right] \\
\hspace{2.5cm}=\operatorname{div}\left\{\alpha^{+}\left[\mu^{+}\left(\nabla u^{+}+\nabla^{t} u^{+}\right)
+\lambda^{+} \operatorname{div} u^{+} \operatorname{Id}\right]\right\}, \\
\partial_{t}\left(R^{-} u^{-}\right)+\operatorname{div}\left(R^{-} u^{-} \otimes u^{-}\right)+\alpha^{-} \mathcal{C}^{2}\left[\rho^{-} \nabla R^{+}+\left(\rho^{+}-\rho^{-} \frac{\alpha^{+} f^{\prime}}{s_{+}^{2}}\right) \nabla R^{-}\right] \\
\hspace{2.5cm}=\operatorname{div}\left\{\alpha^{-}\left[\mu^{-}\left(\nabla
u^{-}+\nabla^{t} u^{-}\right)+\lambda^{-} \operatorname{div} u^{-}
\operatorname{Id}\right]\right\}.
\end{array}\right.
\end{equation}
In the present paper, we consider the initial value problem to
\eqref{1.13} in the whole space $\mathbb R^3$ subject to the initial condition
\begin{equation}\label{1.14} (R^{+}, u^{+}, R^{-}, u^{-})(x,
0)=(R_{0}^{+}, u_{0}^{+}, R_{0}^{-},
u_{0}^{-})(x)\rightarrow(R_{\infty}^{+}, \overrightarrow{0}, R_{\infty}^{-}, \overrightarrow{0}) \quad
\hbox{as}\quad |x|\rightarrow\infty \in \mathbb{R}^{3},
\end{equation}
where $R^{\pm}_\infty>0$ denote the background doping profile, and
for simplicity, are taken as 1 in this paper.
In this work, we investigate the instability of the constant equilibrium state for the Cauchy problem \eqref{1.13}--\eqref{1.14} in the case that $f^{\prime}(1)>0$, which should be kept in mind throughout the rest of this paper. Taking
\[ n^{\pm}=R^{\pm}-1,
\]
then we can rewrite \eqref{1.13} in terms of the varaibles $(n^+, u^+, n^-, u^-)$:
\begin{equation}\label{1.15}
\left\{\begin{array}{l}
\partial_{t} n^++\operatorname{div}u^+=F_1, \\
\partial_{t}u^{+}+\alpha_1\nabla n^++\alpha_2\nabla
n^--\nu^+_1\Delta u^+-\nu^+_2\nabla\operatorname{div} u^+=F_2, \\
\partial_{t} n^-+\operatorname{div}u^-=F_3, \\
\partial_{t}u^{-}+\alpha_3\nabla n^++\alpha_4\nabla
n^--\nu^-_1\Delta u^--\nu^-_2\nabla\operatorname{div} u^-=F_4, \\
\end{array}\right.
\end{equation}
where $\nu_{1}^{\pm}=\frac{\mu^{\pm}}{\rho^{\pm}(1,1)}$,
$\nu_{2}^{\pm}=\frac{\mu^{\pm}+\lambda^{\pm}}{\rho^{\pm}(1,1)}>0$,
$\alpha_{1}=\frac{\mathcal{C}^{2}(1,1)
\rho^{-}(1,1)}{\rho^{+}(1,1)}$,
$\alpha_{2}=\mathcal{C}^{2}(1,1)+\frac{\mathcal{C}^{2}(1,1)
\alpha^{-}(1,1) f^{\prime}(1)}{s_{-}^{2}(1,1)}$,
$\alpha_{3}=\mathcal{C}^{2}(1,1)$,
$\alpha_{4}=\frac{\mathcal{C}^{2}(1,1)
\rho^{+}(1,1)}{\rho^{-}(1,1)}-\frac{\mathcal{C}^{2}(1,1)
\alpha^{+}(1,1) f^{\prime}(1)}{s_{+}^{2}(1,1)}$,
 and
the nonlinear terms are given by

\begin{align}
\label{1.16}F_{1}=&-\operatorname{div}\left(n^{+} u^{+}\right), \\
F_{2}^{i}=&-g_+\left(n^{+}, n^{-}\right) \partial_{i} n^{+}-\bar{g}_{+}\left(n^{+}, n^{-}\right) \partial_{i} n^{-}
-\left(u^{+} \cdot \nabla\right) u_{i}^{+} \nonumber\\
&+\mu^{+} h_{+}\left(n^{+}, n^{-}\right) \partial_{j}
n^{+} \partial_{j} u_{i}^{+}+\mu^{+} k_{+}\left(n^{+}, n^{-}\right) \partial_{j} n^{-} \partial_{j} u_{i}^{+} \nonumber\\
\label{1.17}&+\mu^{+} h_{+}\left(n^{+}, n^{-}\right) \partial_{j}
n^{+} \partial_{i} u_{j}^{+}+\mu^{+} k_{+}\left(n^{+}, n^{-}\right)
\partial_{j} n^{-}
 \partial_{i} u_{j}^{+}\\
&+\lambda^{+} h_{+}\left(n^{+}, n^{-}\right) \partial_{i} n^{+}
\partial_{j} u_{j}^{+}+\lambda^{+} k_{+}\left(n^{+},
n^{-}\right) \partial_{i} n^{-} \partial_{j} u_{j}^{+} \nonumber\\
&+\mu^{+} l_{+}\left(n^{+}, n^{-}\right) \partial_{j}^{2} u_{i}^{+}+\left(\mu^{+}+\lambda^{+}\right) l_{+}\left(n^{+}, n^{-}\right) \partial_{i}
 \partial_{j} u_{j}^{+}, \nonumber\\
\label{1.18}F_{3}=&-\operatorname{div}\left(n^{-} u^{-}\right), \\
F_{4}^{i}=&-g_-\left(n^{+}, n^{-}\right) \partial_{i} n^{-}-
\bar{g}_{-}\left(n^{+}, n^{-}\right) \partial_{i} n^{+}-\left(u^{-} \cdot \nabla\right) u_{i}^{-}\nonumber \\
&+\mu^{-} h_{-}\left(n^{+}, n^{-}\right) \partial_{j} n^{+} \partial_{j} u_{i}^{-}+\mu^{-} k_{-}\left(n^{+}, n^{-}\right)
\partial_{j} n^{-} \partial_{j} u_{i}^{-} \nonumber\\
\label{1.19}&+\mu^{-} h_{-}\left(n^{+}, n^{-}\right) \partial_{j}
n^{+}
\partial_{i} u_{j}^{-}+\mu^{-} k_{-}\left(n^{+}, n^{-}\right)
 \partial_{j} n^{-} \partial_{i} u_{j}^{-} \\
&+\lambda^{-} h_{-}\left(n^{+}, n^{-}\right) \partial_{i} n^{+} \partial_{j} u_{j}^{-}+\lambda^{-} k_{-}\left(n^{+}, n^{-}\right)
 \partial_{i} n^{-} \partial_{j} u_{j}^{-}\nonumber \\
&+\mu^{-} l_{-}\left(n^{+}, n^{-}\right) \partial_{j}^{2}
u_{i}^{-}+\left(\mu^{-}+\lambda^{-}\right) l_{-}\left(n^{+},
n^{-}\right) \partial_{i} \partial_{j} u_{j}^{-},\nonumber
\end{align}

where
\begin{equation}\label{1.20}
\left\{\begin{array}{l}
g_{+}\left(n^{+}, n^{-}\right)=\frac{\left(\mathcal{C}^{2} \rho^{-}\right)\left(n^{+}+1, n^{-}+1\right)}{\rho^{+}\left(n^{+}+1, n^{-}+1\right)}-\frac{\left(\mathcal{C}^{2} \rho^{-}\right)(1,1)}{\rho^{+}(1,1)}, \\
g_{-}\left(n^{+}, n^{-}\right)=\frac{\left(\mathcal{C}^{2} \rho^{+}\right)\left(n^{+}+1, n^{-}+1\right)}{\rho^{-}\left(n^{+}+1, n^{-}+1\right)}-\frac{\left(\mathcal{C}^{2} \rho^{+}\right)(1,1)}{\rho^{-}(1,1)}-\frac{f^{\prime}\left(n^{-}+1\right)\left(\mathcal{C}^{2} \alpha^{+}\right)\left(n^{+}+1, n^{-}+1\right)}{s_{+}^{2}\left(n^{+}+1, n^{-}+1\right)} \\
\hspace{2.2cm}+\frac{f^{\prime}(1)\left(\mathcal{C}^{2} \alpha^{+}\right)(1,1)}{s_{+}^{2}(1,1)}, \\
\end{array}\right.\end{equation}

\begin{equation}\label{1.21}
\left\{\begin{array}{l}
\bar{g}_{+}\left(n^{+}, n^{-}\right)=\mathcal{C}^{2}\left(n^{+}+1, n^{-}+1\right)-=\mathcal{C}^{2}\left(1, 1\right)
+\frac{f^{\prime}\left(n^{-}+1\right)\left(\mathcal{C}^{2} \alpha^{-}\right)\left(n^{+}+1, n^{-}+1\right)}{s_{-}^{2}\left(n^{+}+1, n^{-}+1\right)}\\
\hspace{2.2cm}-\frac{f^{\prime}(1)\left(\mathcal{C}^{2} \alpha^{-}\right)(1,1)}{s_{-}^{2}(1,1)},\\
\bar{g}_{-}\left(n^{+}, n^{-}\right)=\mathcal{C}^{2}\left(n^{+}+1, n^{-}+1\right)-\mathcal{C}^{2}(1,1),\\
\end{array}\right.
\end{equation}

\begin{equation}\label{1.22}
\left\{\begin{array}{l}
h_{+}\left(n^{+}, n^{-}\right)=\frac{\left(\mathcal{C}^{2}\alpha^{-}\right)\left(n^{+}+1, n^{-}+1\right)}{(n^++1)s_{-}^{2}\left(n^{+}+1, n^{-}+1\right)},\\
h_{-}\left(n^{+}, n^{-}\right)=-\frac{\left(\mathcal{C}^{2} \right)\left(n^{+}+1, n^{-}+1\right)}{(\rho^-s_{-}^{2})\left(n^{+}+1, n^{-}+1\right)},
\end{array}\right.
\end{equation}

\begin{equation}\label{1.23}
\left\{\begin{array}{l}
k_{+}\left(n^{+}, n^{-}\right)=-\left[\frac{\mathcal{C}^{2}\left(n^{+}+1, n^{-}+1\right)}{(n^++1)(s_{+}^{2}\rho^+)\left(n^{+}+1, n^{-}+1\right)}+\frac{f^{\prime}(n^-+1)\mathcal{C}^{2}\left(n^{+}+1, n^{-}+1\right)}{(\rho^+\rho^-s_{+}^{2}s_{-}^{2})\left(n^{+}+1, n^{-}+1\right)}\right],\\
k_{-}\left(n^{+}, n^{-}\right)=-\frac{\left(\alpha^+\mathcal{C}^{2}\right)\left(n^{+}+1, n^{-}+1\right)}{(n^-+1)s_{+}^{2}\left(n^{+}+1, n^{-}+1\right)}+\frac{f^{\prime}(n^-+1)\left(\alpha^+\mathcal{C}^{2}\right)\left(n^{+}+1, n^{-}+1\right)}{(\rho^-s_{+}^{2}s_{-}^{2})\left(n^{+}+1, n^{-}+1\right)},\\
\end{array}\right.
\end{equation}

\begin{equation}\label{1.24}
l_{\pm}(n^+, n^-)=\frac{1}{\rho_{\pm}\left(n^{+}+1, n^{-}+1\right)}-\frac{1}{\rho_{\pm}\left(1, 1\right)}.
\end{equation}
Taking change of variables by
\[
n^{+} \rightarrow \alpha_{1} n^{+}, \quad u^{+} \rightarrow \sqrt{\alpha_{1} u^{+}}, \quad n^{-} \rightarrow \alpha_{4} n^{-}, \quad u^{-} \rightarrow \sqrt{\alpha_{4} u^{-}},
\]
and setting
\[
\beta_{1}=\sqrt{\alpha_{1}}, \quad \beta_{2}=\frac{\alpha_{2} \sqrt{\alpha_{1}}}{\alpha_{4}},
\quad \beta_{3}=\frac{\alpha_{3} \sqrt{\alpha_{4}}}{\alpha_{1}}, \quad \beta_{4}=\sqrt{\alpha_{4}}
\]
and
\[
\beta^{+}=\sqrt{\frac{\beta_{1}}{\beta_{2}}}, \quad \beta^{-}=\sqrt{\frac{\beta_{4}}{\beta_{3}}},
\]
the Cauchy problem \eqref{1.13} and \eqref{1.14} can be reformulated as
\begin{equation}\label{1.25}
\left\{\begin{array}{l}
\partial_{t} n^{+}+\beta_{1} \operatorname{div} u^{+}=\mathcal{F}_{1}, \\
\partial_{t} u^{+}+\beta_{1} \nabla n^{+}+\beta_{2} \nabla n^{-}-v_{1}^{+} \Delta u^{+}-v_{2}^{+} \nabla \operatorname{div} u^{+}=\mathcal{F}_{2}, \\
\partial_{t} n^{-}+\beta_{4} \operatorname{div} u^{-}=\mathcal{F}_{3}, \\
\partial_{t} u^{-}+\beta_{3} \nabla n^{+}+\beta_{4} \nabla n^{-}-v_{1}^{-} \Delta u^{-}-v_{2}^{-} \nabla \operatorname{div} u^{-}=\mathcal{F}_{4},
\end{array}\right.
\end{equation}
subject to the initial condition
\begin{equation}\label{1.26}
\left(n^{+}, u^{+}, n^{-}, u^{-}\right)(x, 0)=\left(n_{0}^{+}, u_{0}^{+}, n_{0}^{-}, u_{0}^{-}\right)(x) \rightarrow(0, \overrightarrow{0}, 0, \overrightarrow{0}), \quad \text { as }|x| \rightarrow+\infty,
\end{equation}
where the nonlinear terms are given by
\[
\mathcal{F}_{1}=\alpha_{1} F_{1}\left(\frac{n^{+}}{\alpha_{1}}, \frac{u^{+}}{\sqrt{\alpha_{1}}}\right), \quad \mathcal{F}_{2}=\sqrt{\alpha_{1}} F_{2}
\left(\frac{n^{+}}{\alpha_{1}}, \frac{u^{+}}{\sqrt{\alpha_{1}}}, \frac{n^{-}}{\alpha_{4}}, \frac{u^{-}}{\sqrt{\alpha_{4}}}\right),
\]
and
\[
\mathcal{F}_{3}=\alpha_{4} F_{3}\left(\frac{n^{-}}{\alpha_{4}}, \frac{u^{-}}{\sqrt{\alpha_{4}}}\right), \quad \mathcal{F}_{4}=\sqrt{\alpha_{4}} F_{4}\left(\frac{n^{+}}{\alpha_{1}}, \frac{u^{+}}{\sqrt{\alpha_{1}}}, \frac{n^{-}}{\alpha_{4}}, \frac{u^{-}}{\sqrt{\alpha_{4}}}\right).
\]
Noticing that
\begin{equation}\label{1.27}
\beta_{1} \beta_{4}-\beta_{2} \beta_{3}=-\frac{\mathcal{C}^{2}(1,1) f^{\prime}(1)}{\sqrt{\alpha_{1} \alpha_{4}} \rho^{+}(1,1)}<0,
\end{equation}
it is clear that $\beta^+\beta^-<1$. Before stating our main results, let us state the
corresponding linearized system of \eqref{1.25} as follows:
\begin{equation}\label{1.28}
\left\{\begin{array}{l}
\partial_{t} \tilde n^{+}+\beta_{1} \operatorname{div} \tilde u^{+}=0, \\
\partial_{t}\tilde u^{+}+\beta_{1} \nabla\tilde n^{+}+\beta_{2} \nabla\tilde n^{-}-v_{1}^{+} \Delta\tilde u^{+}-v_{2}^{+} \nabla \operatorname{div}\tilde u^{+}=0, \\
\partial_{t}\tilde n^{-}+\beta_{4} \operatorname{div}\tilde u^{-}=0, \\
\partial_{t}\tilde u^{-}+\beta_{3} \nabla\tilde n^{+}+\beta_{4} \nabla\tilde n^{-}-v_{1}^{-} \Delta\tilde u^{-}-v_{2}^{-} \nabla \operatorname{div}\tilde u^{-}=0.
\end{array}\right.
\end{equation}

\bigskip

\medskip

Now, we are in a position to state our main results. The first one is concerned with
the linear instability, which is stated in the following theorem.
\smallskip
\begin{Theorem}[Linear instability]\label{2mainth} Let $\theta=\frac{\sqrt{(\nu^+\beta_4^2+\nu^-\beta_1^2)^2+4\nu^+\nu^-(\beta_1\beta_2\beta_3\beta_4-\beta_1^2\beta_4^2)}
-(\nu^+\beta_4^2+\nu^-\beta_1^2)}{2\nu^+\nu^-}$ which is positive due to \eqref{1.27}, where $\nu^\pm=\nu_1^\pm+\nu_2^\pm.$ Then for any $\vartheta>0$, the linearized system \eqref{1.28} admits an unstable solution $(\tilde n^+_\vartheta,\tilde u^+_\vartheta,\tilde n^-_\vartheta,\tilde u^-_\vartheta)$  satisfying
	$$\tilde n^\pm_\vartheta \in C^0(0, \infty;  H^{2}(\mathbb{R}^3))\cap C^1(0, \infty; H^{1}(\mathbb{R}^3)),\quad \text{and}\quad \tilde u^\pm_\vartheta\in C^0(0, \infty; H^{2}(\mathbb{R}^3))\cap C^1(0, \infty; L^{2}(\mathbb{R}^3)),$$
and
\begin{equation}\label{1.29}\left\|\tilde n^+_{0,\vartheta}\right\|_{L^2}\left\|\tilde u^+_{0,\vartheta}\right\|_{L^2}\left\|\tilde n^-_{0,\vartheta}\right\|_{L^2}\left\|\tilde u^-_{0,\vartheta}\right\|_{L^2}>0.
\end{equation}
Moreover, the solution satisfies the following estimate:
\begin{equation}\text{e}^{(\theta-\vartheta) t} \|\tilde{n}^\pm_{0,\vartheta}\|_{L^2}\le \|\tilde{n}^\pm_\vartheta(t)\|_{L^2}\le \text{e}^{\theta t}\|\tilde{n}^\pm_{0,\vartheta}\|_{L^2} \quad \text{and}\quad \text{e}^{(\theta-\vartheta) t}\|\tilde{u}^\pm_{0,\vartheta}\|_{L^2}\le\|\tilde{u}^\pm_\vartheta(t)\|_{L^2}\le\text{e}^{\theta t} \|\tilde{u}^\pm_{0,\vartheta}\|_{L^2}.\label{1.30}\end{equation}
\end{Theorem}

\smallskip

\begin{remark} For any $\epsilon>0$ which may be small enough, it is direct to check that $(\epsilon\tilde n^+,\epsilon\tilde u^+,\epsilon\tilde n^-,\epsilon\tilde u^-)$ is still a solution of system \eqref{1.28}. This solution is obvious unstable due to \eqref{1.29} and \eqref{1.30}.
\end{remark}

\smallskip

The second result is  concerned with nonlinear instability, which is stated in the following theorem.

\smallskip

\begin{Theorem}[Nonlinear instability]\label{3mainth} The steady state $(0, \overrightarrow{0}, 0, \overrightarrow{0})$ of the system \eqref {1.25} is unstable in the Hadamard sense, that is, there exist  positive constants $\theta$, $\vartheta$, $\epsilon_0$ and $\delta_0$, and functions $( \tilde n_{0,\vartheta}^+,\tilde u_{0,\vartheta}^+,\tilde n_{0,\vartheta}^-,\tilde u_{0,\vartheta}^-)\in H^4(\mathbb R^3)$, such that for any $\epsilon\in(0,\epsilon_0)$ and the initial data
\begin{equation}\label{1.32}(  n_0^+, u_0^+, n_0^-,u_0^-)\triangleq\epsilon( \tilde n_{0,\vartheta}^+,\tilde u_{0,\vartheta}^+,\tilde n_{0,\vartheta}^-,\tilde u_{0,\vartheta}^-),\end{equation}
the Cauchy problem \eqref{1.25} and \eqref{1.32} admits a unique strong solution satisfying
	$$ n^\pm \in C^0(0, T^{\max};  H^{4}(\mathbb{R}^3))\cap C^1(0, T^{\max}; H^{3}(\mathbb{R}^3))\quad \text{and}\quad  u^\pm\in C^0(0, T^{\max}; H^{4}(\mathbb{R}^3))\cap C^1(0, T^{\max}; H^{2}(\mathbb{R}^3)),$$
and
\begin{equation}\label{1.33}\left\| (n^+,u^+,n^-,u^-)(T^\varepsilon)\right\|_{H^4}\ge \delta_0.
\end{equation}
for some escape time $T^\varepsilon\in [0,T^{\max})$, where $T^{\max}$ denotes the maximal time of existence of the solution.
\end{Theorem}

\smallskip

\begin{remark} Theorem \ref{2mainth} and Theorem \ref{3mainth} show that for the case $f'(1)>0$, the constant equilibrium state of the two--fluid model is linearly globally unstable and nonlinearly locally unstable
in the sense of Hadamard, which is in contrast to
the cases $f'(1)<0$ in Evje--Wang--Wen \cite{Evje9} and $P^+=P^-$ (corresponding to $f'(1)=0$) in Wu--Yao--Zhang \cite{WYZ} where the constant equilibrium state of the two--fluid model \eqref{1.5} is nonlinearly globally stable.
\end{remark}

Now, let us sketch the main ideas in the proofs of Theorem \ref{2mainth} and Theorem \ref{3mainth}.
For the proof of Theorem \ref{2mainth}, we need construct a solution to the linearized system \eqref{1.28} that has a growing $H^k$
norm for any $k$ and the proof can be outlined as follows.
First, we exclude the stabilizing part of the linearized system by employing the Hodge decomposition technique firstly
introduced by Danchin \cite{Dan1} to split the linearized system into three systems (see \eqref {2.1} and \eqref {2.2} for details). One is a $4\times 4$ system and
its characteristic polynomial possesses four distinct roots, the other two systems are the heat equation.
This key observation allows us to construct an unstable solution.
Second, we assume a growing mode ansatz, i.e., $$\widehat{\tilde{n}^+}=\text{e}^{\lambda(|\xi|)t}\widehat{\tilde{n}^+_0},\
\widehat{\tilde{\varphi}^+}=\text{e}^{\lambda(|\xi|)t}\widehat{\tilde{\varphi}^+_0},\
\widehat{\tilde{n}^-}=\text{e}^{\lambda(|\xi|)t}\widehat{\tilde{n}^-_0},\
\widehat{\tilde{\varphi}^-}=\text{e}^{\lambda(|\xi|)t}\widehat{\tilde{\varphi}^-_0}, ~~\hbox{for some}~\lambda,$$
and submit this ansatz into the Fourier transformation of the $4\times 4$ system to get
a time--independent system for $\lambda$.
Third, we solve the time--independent system by making careful analysis and using several key observations. Indeed, noticing that
the characteristic polynomial $F(\lambda)$ defined in \eqref{2.6} is a strictly increasing function on $(0, \infty)$, and
$F(\theta)>0$ for $\theta>0$ defined in Theorem \ref{2mainth}, we show that $0<\lambda_1<\theta$ is the unique positive root of the characteristic equation $F(\lambda)=0$, and $\theta>0$ in Theorem \ref{2mainth} is the largest possible growth rate since $Re(\lambda_i)\leq \theta$ with $1\leq i\leq 4$. Therefore, the growing mode constructed in Theorem \ref{2mainth} actually does grow in time at
the fastest possible rate.\par
For the proof of Theorem \ref{3mainth}, we deduce the nonlinear instability. Compared to \cite{Guo1,Jang,Jiang1,Jiang2,WangT} where nonlinear energy estimates and a careful bootstrap argument are employed to prove stability and instability, we need to develop new
ingredients in the proof to handle with the difficulties arising from the strong interaction of two fluids,
which requires some new thoughts. Indeed, since the strong
coupling terms are involved in the right--hand of the system \eqref{1.25}, it seems impossible to follow the energy methods of \cite{Guo1,Jang,Jiang1,Jiang2,WangT}
to get the lyapunov--type inequality: $\frac{d}{dt}\mathcal{E}(t)\leq \theta \mathcal{E}(t)$ to prove the largest possible growth rate.
Therefore,  we must pursue another route by resorting to semigroup methods to capture the largest possible growth rate, but the cost is that
we need the higher regularity of the solutions. More precisely,
with the help of the global linear instability result of Theorem \ref{2mainth} and a local existence theorem of classical solutions to the original nonlinear system, we can make delicate spectral analysis for the linearized system and apply Duhamel's principle to prove the nonlinear instability
result stated in Theorem \ref{2mainth}.
\subsection{Notations and conventions.}
Throughout this paper, we denote $H^k(\mathbb R^3)$ by the usual Sobolev spaces with norm $\|\cdot\|_{H^k}$ and
denote $L^p$, $1\leq p\leq \infty$ by the usual $L^p(\mathbb R^3)$
spaces with norm $\|\cdot\|_{L^p}$. We drop the domain $\mathbb R^3$ in integrands over $\mathbb R^3$. For the sake of conciseness, we
do not precise in functional space names when they are concerned
with scalar--valued or vector--valued functions, $\|(f, g)\|_X$
denotes $\|f\|_X+\|g\|_X$.  We will employ the notation $a\lesssim
b$ to mean that $a\leq Cb$ for a universal constant $C>0$ that only
depends on the parameters coming from the problem. We denote
$\nabla=\partial_x=(\partial_1,\partial_2,\partial_3)$, where
$\partial_i=\partial_{x_i}$, $\nabla_i=\partial_i$ and put
$\partial_x^\ell f=\nabla^\ell f=\nabla(\nabla^{\ell-1}f)$.  Let
$\Lambda^s$ be the pseudo differential operator defined by
\begin{equation}\Lambda^sf=\mathfrak{F}^{-1}(|{\bf \xi}|^s\widehat f),~\hbox{for}~s\in \mathbb{R},\nonumber\end{equation}
where $\widehat f$ and $\mathfrak{F}(f)$ are the Fourier transform
of $f$.

\smallskip

\bigskip
\section{\leftline {\bf{Linear instability.}}}
\setcounter{equation}{0}

To construct a solution to the linearized system \eqref{1.28} that has growing $H^k$--norm for any positive integer $k$, by using a real method
as in \cite{Kowalczyk}, one need to make a detailed
analysis on the properties of the semigroup. To exclude the stabilizing part, we will employ the Hodge decomposition technique firstly
introduced by Danchin \cite{Dan1} to split the linear system into three systems. One only has four
equations and its characteristic polynomial possesses four distinct roots, the other two systems
are the heat equation. This key observation allows us to construct a unstable solution.
To see this, let
$\varphi^{\pm}=\Lambda^{-1}{\rm div}\tilde{u}^{\pm}$
be the ``compressible part" of the velocities $\tilde{u}^{\pm}$, and denote $\phi^{\pm}=\Lambda^{-1}{\rm curl}\tilde{u}^{\pm}$ (with $({\rm curl} z)_i^j
=\partial_{x_j}z^i-\partial_{x_i}z^j$) by the ``incompressible part"
of the velocities $\tilde{u}^{\pm}$. Setting  $\nu^{\pm}=\nu^{\pm}_1+\nu^{\pm}_2$, the system \eqref{1.28} can be
decomposed into
the following three systems:
\begin{equation}\label{2.1}
\begin{cases}
\partial_t{\tilde{n}^+}+\beta_1\Lambda{\varphi^+}=0,\\
\partial_t{\varphi^+}-\beta_1\Lambda{\tilde{n}^+}-\beta_2\Lambda{\tilde{n}^-}+\nu^+\Lambda^2{\varphi^+}=0,\\
\partial_t{\tilde{n}^-}+\beta_4\Lambda{\varphi^-}=0,\\
\partial_t{\varphi^-}-\beta_3\Lambda{\tilde{n}^+}-\beta_4\Lambda{\tilde{n}^-}+\nu^-\Lambda^2{\varphi^-}=0,\\
\end{cases}
\end{equation}
and
\begin{equation}\label{2.2}
\begin{cases}
\partial_t\phi^++\nu^+_1\Lambda^2\phi^+=0,\\
\partial_t\phi^-+\nu^-_1\Lambda^2\phi^-=0.
\end{cases}
\end{equation}
We see that Eqs. \eqref{2.2}$_1$ and \eqref{2.2}$_2$ are the standard parabolic equations with good stability. Thus, the onset of instabilities of system \eqref{1.28} comes from \eqref{2.1}. Taking the Fourier transform to the system \eqref{2.1}, one has
\begin{equation}\label{2.3}
\begin{cases}
\partial_t\widehat{\tilde{n}^+}+\beta_1|\xi|\widehat{\varphi^+}=0,\\
\partial_t\widehat{\varphi^+}-\beta_1|\xi|\widehat{\tilde{n}^+}-\beta_2|\xi|\widehat{\tilde{n}^-}+\nu^+|\xi|^2\widehat{\varphi^+}=0,\\
\partial_t\widehat{\tilde{n}^-}+\beta_4|\xi|\widehat{\varphi^-}=0,\\
\partial_t\widehat{\varphi^-}-\beta_3|\xi|\widehat{\tilde{n}^+}-\beta_4|\xi|\widehat{\tilde{n}^-}+\nu^-|\xi|^2\widehat{\varphi^-}=0.\\
\end{cases}
\end{equation}
To construct a solution to the linearized equations \eqref{2.3} that has growing $H^k$--norm for any $k$, we shall make a growing normal mode ansatz of solutions, i.e.,
$$\widehat{\tilde{n}^+}=\text{e}^{\lambda(|\xi|)t}\widehat{\tilde{n}^+_0},\
\widehat{\tilde{\varphi}^+}=\text{e}^{\lambda(|\xi|)t}\widehat{\tilde{\varphi}^+_0},\
\widehat{\tilde{n}^-}=\text{e}^{\lambda(|\xi|)t}\widehat{\tilde{n}^-_0},\
\widehat{\tilde{\varphi}^-}=\text{e}^{\lambda(|\xi|)t}\widehat{\tilde{\varphi}^-_0}.$$
Substituting this ansazt into \eqref{2.3}, one obtains the time--independent system
\begin{equation}\label{2.4}
\begin{cases}
\lambda{\widehat{\tilde{n}^+_0}}+\beta_1|\xi|\widehat{{\varphi^+_0}}=0,\\
\lambda\widehat{{\varphi^+_0}}-\beta_1|\xi|\widehat{{\tilde{n}^+_0}}-\beta_2|\xi|\widehat{{\tilde{n}^-_0}}+\nu^+|\xi|^2\widehat{{\varphi^+_0}}=0,\\
\lambda\widehat{{\tilde{n}^-_0}}+\beta_4|\xi|\widehat{{\varphi^-_0}}=0,\\
\lambda\widehat{{\varphi^-_0}}-\beta_3|\xi|\widehat{{\tilde{n}^+_0}}-\beta_4|\xi|\widehat{{\tilde{n}^-_0}}+\nu^-|\xi|^2\widehat{{\varphi^-_0}}=0.\\
\end{cases}
\end{equation}
After a series of tedious but direct calculations, we can conclude from \eqref{2.4} that
\begin{equation}\begin{split}\label{2.5}
&[\lambda^4+(\nu^+|\xi|^2+\nu^-|\xi|^2)\lambda^3+(\beta_1^2|\xi|^2+\beta_4^2|\xi|^2+\nu^+\nu^-|\xi|^4)\lambda^2\\&+(\nu^+\beta_4^2|\xi|^4+\nu^-\beta_1^2|\xi|^4)\lambda
+\beta_1^2\beta_4^2|\xi|^4-\beta_1\beta_2\beta_3\beta_4|\xi|^4]\widehat{{\varphi^-_0}}=0.
\end{split}\end{equation}
Therefore, the system \eqref{2.4} has non--zero solutions if the characteristic equation
\begin{equation}\begin{split}\label{2.6}
F(\lambda)=&\lambda^4+(\nu^+|\xi|^2+\nu^-|\xi|^2)\lambda^3+(\beta_1^2|\xi|^2+\beta_4^2|\xi|^2+\nu^+\nu^-|\xi|^4)\lambda^2\\&+(\nu^+\beta_4^2|\xi|^4+\nu^-\beta_1^2|\xi|^4)\lambda
+\beta_1^2\beta_4^2|\xi|^4-\beta_1\beta_2\beta_3\beta_4|\xi|^4=0
\end{split}\end{equation}
has a real characteristic root.
\begin{Lemma}\label{lemma2.1}
There exists a positive constant $\eta_1\gg 1 $, such that for
$|\xi|\ge \eta_1$, the characteristic equation \eqref{2.6} admits a real positive solution satisfying  the following Taylor series
expansion
\begin{equation}\lambda_1=\theta+\mathcal {O}(|\xi|^{-1}).\label{2.7}\end{equation}
Moreover, the following estimate holds
\begin{equation}\lambda_1<\theta \quad \text{for any}\quad \xi\in\mathbb R^3.\label{2.8}\end{equation}
\end{Lemma}
\begin{proof}Employing
the similar argument of Taylor series expansion as in
\cite{Mat1}, then \eqref{2.7} follows from some tedious but direct calculations. It is noticed that
$F(\lambda)$ is a strictly monotonically increasing function if $\lambda>0$. Furthermore,
\begin{equation}\nonumber\begin{split}F(\theta)>\nu^+\nu^-|\xi|^4\theta^2+(\nu^+\beta_4^2|\xi|^4+\nu^-\beta_1^2|\xi|^4)\theta
+\beta_1^2\beta_4^2|\xi|^4-\beta_1\beta_2\beta_3\beta_4|\xi|^4=0,\end{split}\end{equation}
therefore \eqref{2.8} holds and the proof of lemma is completed.
\end{proof}

\smallskip

 Let $\phi\in C_0^\infty(\mathbb
R^3_{{\bf \xi}})$ be a radial function satisfying $\phi({\bf \xi})=1$ when $\frac{3}{2}\eta\le |{\bf
\xi}|\leq 3\eta$ and $\phi({\bf \xi})=0$ when $|{\bf
\xi}|\le \eta$ and $|{\bf
\xi}|\ge 4\eta$. From \eqref{2.4},  we set
$${\widehat{\tilde{n}^+_0}}=\phi({\bf \xi}),\ \widehat{{\varphi^+_0}}=-\frac{\lambda_1(|\xi|)}{\beta_1|\xi|}\phi({\bf \xi}),\ {\widehat{\tilde{n}^-_0}}=-\frac{\lambda^2_1(|\xi|)+\beta_1^2|\xi|^2+\nu^+\lambda_1(|\xi|)|\xi|^2}{\beta_1\beta_2|\xi|^2}\phi({\bf \xi})$$
and
$$\widehat{{\varphi^-_0}}=\frac{\lambda^3_1(|\xi|)+\beta_1^2\lambda_1(|\xi|)|\xi|^2+\nu^+\lambda^2_1(|\xi|)|\xi|^2}{\beta_1\beta_2\beta_4|\xi|^3}\phi({\bf \xi}).$$
Then, it is direct to check that $(\widehat{\tilde{n}^+_0},\widehat{\tilde\varphi^+_0},\widehat{\tilde{n}^-_0},\widehat{\tilde\varphi^-_0})$ is a solution of the system \eqref{2.4}. Thus,  we conclude the following proposition, which implies Theorem \ref{2mainth}.

\begin{Proposition}\label{Prop2.2} Let
$${\tilde{n}^\pm}=\mathfrak{F}^{-1}\left(\text{e}^{\lambda_1t}\widehat{\tilde{n}^\pm_0}\right)\quad \text{and}\quad
{\tilde{u}^\pm}=-\Lambda^{-1}\nabla\mathfrak{F}^{-1}\left(\text{e}^{\lambda_1(|\xi|)t}\widehat{\tilde{\varphi}^\pm_0}\right).$$
Then $(\tilde{n}^+,\tilde u^+,\tilde{n}^-,\tilde u^-)$ is a solution of \eqref{1.29} and
satisfies
\begin{equation}\text{e}^{(\theta-\vartheta) t} \|\tilde{n}^\pm_0\|_{L^2}\le \|\tilde{n}^\pm(t)\|_{L^2}\le \text{e}^{\theta t}\|\tilde{n}^\pm_0\|_{L^2} \quad \text{and}\quad \text{e}^{(\theta-\vartheta) t}\|\tilde{u}^\pm_0\|_{L^2}\le\|\tilde{u}^\pm(t)\|_{L^2}\le\text{e}^{\theta t} \|\tilde{u}^\pm_0\|_{L^2},\label{2.9}\end{equation}
if $\eta_1$ large enough.
\end{Proposition}
\begin{proof} Set $\phi^\pm\equiv0$. As the definition of $\varphi^\pm$ and $\phi^\pm$, and the relation
$$\tilde u^\pm=-\Lambda^{-1}\nabla\varphi^\pm-\Lambda^{-1}\text{div}\phi^\pm,$$
it is easy to prove that $(\tilde{n}^+,\tilde u^+,\tilde{n}^-,\tilde u^-)$ is a solution of \eqref{1.29}. Moreover, in virtue of Plancherel theorem, we have
 \begin{equation}\begin{split}\|\tilde u^\pm(t)\|_{L^2}^2=&\|\widehat{\tilde u^\pm}(t)\|_{L^2}^2\\
 =&\int\text{e}^{2\lambda_1(|\xi|)t}|\widehat{\tilde u^\pm_0}|^2\mathrm{d}\xi\\
 =&\int_{\eta\le|\xi|\le 4|\eta|}\text{e}^{2\lambda_1(|\xi|)t}|\widehat{\tilde u^\pm_0}|^2\mathrm{d}\xi\\
 \ge &~\text{e}^{2(\theta-\vartheta)  t}\|\tilde u^\pm_0(t)\|_{L^2}^2,
 \end{split}\end{equation}
 if $\eta$ is large enough. Performing the similar procedures, we can prove $\|\tilde{u}^\pm(t)\|_{L^2}\le\text{e}^{\theta t} \|\tilde{u}^\pm_0\|_{L^2}$ and  $\text{e}^{(\theta-\vartheta)  t} \|\tilde{n}^\pm_0\|_{L^2}\le\|\tilde{n}^\pm(t)\|_{L^2}\le\text{e}^{\theta t} \|\tilde{n}^\pm_0\|_{L^2}$. The proof of proposition is complete.
\end{proof}

\bigskip

\section{Spectral analysis and linear $L^2$--estimates}\label{1section_appendix}

In this section, we are devoted to  deriving the linear $L^2$--estimates,  by using a real method
as in \cite{Mat1}, one need to make a detailed analysis on the
properties of the semigroup.

\subsection{Spectral analysis for system \eqref{2.1}} We consider the Cauchy problem of \eqref{2.1} with the initial data
 \begin{equation}\label{3.1}(\tilde{n}^+, \varphi^+, \tilde{n}^-, \varphi^-)\big|_{t=0}=({n}^+_0, \Lambda^{-1}{\rm div}\tilde{u}^{+}_0, {n}^-_0, \Lambda^{-1}{\rm div}\tilde{u}^{-}_0)(x)
 \end{equation}
 In terms of the semigroup theory, we may represent the IVP \eqref{2.1} and \eqref{3.1} for $\mathcal U=(\tilde{n}^+, \varphi^+, \tilde{n}^-, \varphi^-)^t$ as
\begin{equation}
\begin{cases}
\mathcal U_t=\mathcal B_1\mathcal U,\\
\mathcal U\big|_{t=0}=\mathcal U_0,
\end{cases}   \label{3.2}
\end{equation}
where the operator $\mathcal B_1$ is defined
by
\begin{equation}\nonumber\mathcal B_1=\begin{pmatrix}
0&-\beta_1\Lambda&0&0\\
\beta_1\Lambda&-\nu^+\Lambda^2&\beta_2\Lambda&0\\
0&0&0&-\beta_4\Lambda\\
\beta_3\Lambda&0&\beta_4\Lambda&-\nu^-\Lambda^2
\end{pmatrix}.\end{equation}
Taking the Fourier transform to the system \eqref{3.2}, we obtain
\begin{equation}
\begin{cases}
\widehat {\mathcal U}_t=\mathcal A_1(\xi)\widehat {\mathcal U},\\
\widehat {\mathcal U}\big|_{t=0}=\widehat {\mathcal U}_0,
\end{cases}   \label{3.3}
\end{equation}
where $\widehat {\mathcal U}(\xi,t)=\mathfrak{F}({\mathcal U}(x,t))$  and $\mathcal A_1(\xi)$ is given by
\begin{equation}\nonumber\mathcal A_1(\xi)=\begin{pmatrix}
0&-\beta|\xi|&0&0\\
\beta_1|\xi|&-\nu^+|\xi|^2&\beta_2|\xi|&0\\
0&0&0&-\beta_4|\xi|\\
\beta_3|\xi|&0&\beta_4|\xi|&-\nu^-|\xi|^2
\end{pmatrix}.\end{equation}
We compute the eigenvalues of the matrix $\mathcal A_1(\xi)$  from the determinant
\begin{equation}\begin{split}\label{3.4}&{\rm det}(\lambda{\rm I}-\mathcal A_1(\xi))\\
&=\lambda^4+(\nu^+|\xi|^2+\nu^-|\xi|^2)\lambda^3+(\beta_1^2|\xi|^2+\beta_4^2|\xi|^2+\nu^+\nu^-|\xi|^4)\lambda^2+(\nu^+\beta_4^2|\xi|^4+\nu^-\beta_1^2|\xi|^4)\lambda\\
&\quad+\beta_1^2\beta_4^2|\xi|^4-\beta_1\beta_2\beta_3\beta_4|\xi|^4\\
&=0,
\end{split}\end{equation}
which is the same as characteristic equation \eqref{2.6} and implies that the matrix $\mathcal A_1(\xi)$ possesses four different
eigenvalues:
\begin{equation}\nonumber
 \lambda_1=\lambda_1(|\xi|),\quad \lambda_2=\lambda_2(|\xi|),\quad
 \lambda_3=\lambda_3(|\xi|),\quad \lambda_4=\lambda_4(|\xi|).
\end{equation}
Consequently, the semigroup $e^{t\mathcal A_1}$ can be decomposed
into
\begin{equation}\label{2.21}
 \text{e}^{t\mathcal A_1(\xi)}=\sum_{i=1}^4\text{e}^{\lambda_it}P_i(\xi),
\end{equation}
where the projector $P_i(\xi)$ is defined by
\begin{equation}\label{2.22}
 P_i(\xi)=\prod_{j\neq i}\frac{\mathcal A_1(\xi)-\lambda_jI}{\lambda_i-\lambda_j}, \quad i,j=1,2,3,4.
\end{equation}
Thus, the solution of IVP \eqref{3.3} can be expressed as
\begin{equation}
\widehat {\mathcal U}(\xi,t)=\text{e}^{t\mathcal A_1(\xi)}\widehat {\mathcal U}_0(\xi)=\left(\sum_{i=1}^4
\text{e}^{\lambda_it}P_i(\xi)\right)\widehat {\mathcal U}_0(\xi).\label{2.23}
\end{equation}

To derive long time properties of  the semigroup
$\text{e}^{t\mathcal A_1}$ in $L^2$--framework, one need to analyze
the asymptotical expansions of $\lambda_i$, $P_i$ $(i =1, 2, 3, 4)$
and $\text{e}^{t\mathcal A_1(\xi)}$. Employing the similar argument of Taylor series expansion as in
\cite{Mat1}, we have the following lemmas from tedious
calculations.

\begin{Lemma}\label{lemma3.1}
There exists a positive constant $\eta_2\ll 1 $ such that, for
$|\xi|\leq \eta_2$, the spectral has the following Taylor series
expansion:
\begin{equation}\label{3.8-1}
\left\{\begin{array}{lll}\displaystyle \lambda_1=-\left[\frac{\nu^++\nu^-}{4}-\frac{\nu^+(\beta_1^2-\beta_4^2)+\nu^-(\beta_4^2-\beta_1^2)}{8\kappa_1}\right]|\xi|^2+\sqrt{\kappa_1-\kappa_2}|\xi|+\mathcal O(|\xi|^3),\\
\displaystyle\lambda_2=-\left[\frac{\nu^++\nu^-}{4}-\frac{\nu^+(\beta_1^2-\beta_4^2)+\nu^-(\beta_4^2-\beta_1^2)}{8\kappa_1}\right]|\xi|^2-\sqrt{\kappa_1-\kappa_2}|\xi|+\mathcal O(|\xi|^3),
\\ \displaystyle  \lambda_3=-\left[\frac{\nu^++\nu^-}{4}+\frac{\nu^+(\beta_1^2-\beta_4^2)+\nu^-(\beta_4^2-\beta_1^2)}{8\kappa_1}\right]|\xi|^2
+\sqrt{\kappa_2+\kappa_1}\text{i}|\xi|+\mathcal O(|\xi|^3),
\\ \displaystyle  \lambda_4=-\left[\frac{\nu^++\nu^-}{4}+\frac{\nu^+(\beta_1^2-\beta_4^2)+\nu^-(\beta_4^2-\beta_1^2)}{8\kappa_1}\right]|\xi|^2-\sqrt{\kappa_2+\kappa_1}\text{i}|\xi|+\mathcal O(|\xi|^3),
\end{array}\right.
\end{equation}
where $\kappa_1=\sqrt{\frac{(\beta_1^2-\beta_4^2)^2}{4}+\beta_1\beta_2\beta_3\beta_4}$ and $\displaystyle\kappa_2=\frac{\beta_1^2+\beta_4^2}{2}$.
\end{Lemma}

For $|\xi|\leq\eta_2$, from Lemma \ref{lemma3.1}, a direct computation gives
{\small \begin{equation}\label{3.9}\begin{split}P_1(\xi)=&\begin{pmatrix}
\frac{2\kappa_1+\beta_4^2-\beta_1^2}{8\kappa_1}&
\frac{\beta_1(2\kappa_1+\beta_4^2-\beta_1^2)}{8\kappa_1\sqrt{\kappa_1-\kappa_2}}&\frac{-\beta_1\beta_2}{4\kappa_1}&\frac{-\beta_1\beta_2\beta_4}{4\kappa_1\sqrt{\kappa_1-\kappa_2}}\\
\frac{\beta_1(\beta_1^2-\beta_4^2-2\kappa_1)+2\beta_2\beta_3\beta_4}{8\kappa_1
\sqrt{\kappa_1-\kappa_2}}&\frac{2\kappa_1+\beta_4^2-\beta_1^2}{8\kappa_1}&-\frac{\beta_2\sqrt{\kappa_1-\kappa_2}}{4\kappa_1}&\frac{-\beta_2\beta_4}{4\kappa_1}\\
\frac{-\beta_3\beta_4}{4\kappa_1}&\frac{-\beta_1\beta_3\beta_4}{4\kappa_1\sqrt{\kappa_1-\kappa_2}}&\frac{2\kappa_1+\beta_1^2-\beta_4^2}{8\kappa_1}&\frac{\beta_4(2\kappa_1
+\beta_1^2-\beta_4^2)}{8\kappa_1\sqrt{\kappa_1-\kappa_2}}\\
-\frac{\beta_3\sqrt{\kappa_1-\kappa_2}}{4\kappa_1}&\frac{-\beta_1\beta_3}{4\kappa_1}&\frac{\beta_4(\beta_4^2-\beta_1^2-2\kappa_1)+2\beta_1\beta_2\beta_3}{8\kappa_1\sqrt{\kappa_1-
\kappa_2}}&\frac{2\kappa_1+\beta_1^2-\beta_4^2}{8\kappa_1}
\end{pmatrix}+\mathcal O(|\xi|),\end{split}\end{equation}

\begin{equation}\label{3.10}\begin{split}P_2(\xi)=&\begin{pmatrix}
\frac{2\kappa_1+\beta_4^2-\beta_1^2}{8\kappa_1}&\frac{-\beta_1(2\kappa_1+\beta_4^2-\beta_1^2)}{8\kappa_1\sqrt{\kappa_1-\kappa_2}}&
\frac{-\beta_1\beta_2}{4\kappa_1}&\frac{\beta_1\beta_2\beta_4}{4\kappa_1\sqrt{\kappa_1-\kappa_2}}\\
-\frac{\beta_1(\beta_1^2-\beta_4^2-2\kappa_1)+2\beta_2\beta_3\beta_4}{8\kappa_1\sqrt{\kappa_1-\kappa_2}}&
\frac{2\kappa_1+\beta_4^2-\beta_1^2}{8\kappa_1}&\frac{\beta_2\sqrt{\kappa_1-\kappa_2}}{4\kappa_1}&\frac{-\beta_2\beta_4}{4\kappa_1}\\
\frac{-\beta_3\beta_4}{4\kappa_1}&\frac{\beta_1\beta_3\beta_4}{4\kappa_1\sqrt{\kappa_1-\kappa_2}}&\frac{2\kappa_1+\beta_1^2-\beta_4^2}
{8\kappa_1}&-\frac{\beta_4(2\kappa_1+\beta_1^2-\beta_4^2)}{8\kappa_1\sqrt{\kappa_1-\kappa_2}}\\
\frac{\beta_3\sqrt{\kappa_1-\kappa_2}}{4\kappa_1}&\frac{-\beta_1\beta_3}{4\kappa_1}&-\frac{\beta_4(\beta_4^2-\beta_1^2-2\kappa_1)+2\beta_1\beta_2\beta_3}
{8\kappa_1\sqrt{\kappa_1-\kappa_2}}&\frac{2\kappa_1+\beta_1^2-\beta_4^2}{8\kappa_1}
\end{pmatrix}+\mathcal O(|\xi|),\end{split}\end{equation}

\begin{equation}\label{3.11}\begin{split}P_3(\xi)=&\begin{pmatrix}
\frac{2\kappa_1+\beta_1^2-\beta_4^2}{8\kappa_1}&\frac{\beta_1(2\kappa_1+\beta_1^2-\beta_4^2)}
{8\kappa_1\sqrt{\kappa_2+\kappa_1}}{i}&\frac{\beta_1\beta_2}{4\kappa_1}&\frac{\beta_1\beta_2\beta_4}{4\kappa_1\sqrt{\kappa_2+\kappa_1}}{i}\\
-\frac{\beta_1(\beta_1^2-\beta_4^2+2\kappa_1)+2\beta_2\beta_3\beta_4}{8\kappa_1\sqrt{\kappa_2+\kappa_1}}{i}&\frac{2\kappa_1+\beta_1^2-\beta_4^2}
{8\kappa_1}&-\frac{\beta_2\sqrt{\kappa_2+\kappa_1}}{4\kappa_1}{i}&\frac{\beta_2\beta_4}{4\kappa_1}\\
\frac{\beta_3\beta_4}{4\kappa_1}&\frac{\beta_1\beta_3\beta_4}{4\kappa_1\sqrt{\kappa_2+\kappa_1}}{i}&
\frac{2\kappa_1+\beta_4^2-\beta_1^2}{8\kappa_1}&\frac{\beta_4(2\kappa_1+\beta_4^2-\beta_1^2)}{8\kappa_1\sqrt{\kappa_2+\kappa_1}}{i}\\
-\frac{\beta_3\sqrt{\kappa_2+\kappa_1}}{4\kappa_1}{i}&\frac{\beta_1\beta_3}{4\kappa_1}&-\frac{\beta_4(\beta_4^2-\beta_1^2+2\kappa_1)
+2\beta_1\beta_2\beta_3}{8\kappa_1\sqrt{\kappa_2+\kappa_1}}{i}&\frac{2\kappa_1+\beta_4^2-\beta_1^2}{8\kappa_1}
\end{pmatrix}+\mathcal O(|\xi|),\end{split}\end{equation}

and
\begin{equation}\label{3.12}\begin{split}P_4(\xi)=&\begin{pmatrix}
\frac{2\kappa_1+\beta_1^2-\beta_4^2}{8\kappa_1}&-\frac{\beta_1(2\kappa_1+\beta_1^2-\beta_4^2)}
{8\kappa_1\sqrt{\kappa_2+\kappa_1}}{i}&\frac{\beta_1\beta_2}{4\kappa_1}&-\frac{\beta_1\beta_2\beta_4}{4\kappa_1\sqrt{\kappa_2+\kappa_1}}{i}\\
\frac{\beta_1(\beta_1^2-\beta_4^2+2\kappa_1)+2\beta_2\beta_3\beta_4}{8\kappa_1\sqrt{\kappa_2+\kappa_1}}{i}&\frac{2\kappa_1+\beta_1^2-\beta_4^2}
{8\kappa_1}&\frac{\beta_2\sqrt{\kappa_2+\kappa_1}}{4\kappa_1}{i}&\frac{\beta_2\beta_4}{4\kappa_1}\\
\frac{\beta_3\beta_4}{4\kappa_1}&-\frac{\beta_1\beta_3\beta_4}{4\kappa_1\sqrt{\kappa_2+\kappa_1}}{i}&
\frac{2\kappa_1+\beta_4^2-\beta_1^2}{8\kappa_1}&-\frac{\beta_4(2\kappa_1+\beta_4^2-\beta_1^2)}{8\kappa_1\sqrt{\kappa_2+\kappa_1}}{i}\\
\frac{\beta_3\sqrt{\kappa_2+\kappa_1}}{4\kappa_1}{i}&\frac{\beta_1\beta_3}{4\kappa_1}&\frac{\beta_4(\beta_4^2-\beta_1^2+2\kappa_1)+2\beta_1\beta_2\beta_3}
{8\kappa_1\sqrt{\kappa_2+\kappa_1}}{i}&\frac{2\kappa_1+\beta_4^2-\beta_1^2}{8\kappa_1}
\end{pmatrix}+\mathcal O(|\xi|),\end{split}\end{equation}}

\begin{Lemma}\label{lemma3.2}
For $\eta_2\le |\xi|\le \eta_1$, there exists a positive constant $C$ such that
\begin{equation}\label{3.13}\text{Re}(\lambda_i)\le \theta \quad \text{and} \quad \left | P_i\right |\le C,
\end{equation}
for $1\le i\le 4.$
\end{Lemma}

\begin{Lemma}\label{lemma3.3}
There exists a positive constants $\eta_1\gg 1 $ such that, for
$|\xi|\geq \eta_1$, the spectral has the following Taylor series
expansion:
\begin{equation}\label{3.8}
\left\{\begin{array}{lll}\displaystyle \lambda_1=\theta+\mathcal O(|\xi|^{-1}),\\
\displaystyle\lambda_2=\frac{-(\nu^+\beta_4^2+\nu^-\beta_1^2)-\kappa_3
}{2\nu^+\nu^-}+\mathcal O(|\xi|^{-1}),
\\ \displaystyle  \lambda_3=-\nu^+|\xi|^2+\frac{\beta_1^2}{\nu^+}
+\mathcal O(|\xi|^{-1}),
\\ \displaystyle  \lambda_4=-\nu^-|\xi|^2+\frac{\beta_4^2}{\nu^-}+\mathcal O(|\xi|^{-1}),
\end{array}\right.
\end{equation}
where $\kappa_3=\sqrt{(\nu^+\beta_4^2+\nu^-\beta_1^2)^2+4\nu^+\nu^-(\beta_1\beta_2\beta_3\beta_4-\beta_1^2\beta_4^2)}$.
\end{Lemma}

 For $|\xi|\geq\eta_1$, from Lemma \ref{lemma3.3}, a direct computation gives
 \begin{equation}\label{3.15}\begin{split}P_1(\xi)=&\begin{pmatrix}
\frac{\nu^+\beta_4^2-\nu^-\beta_1^2+\kappa_3}{2\kappa_3}&
0&-\frac{\beta_1\beta_2\nu^-}{\kappa_3}&0\\
0&0&0&0\\
\frac{-\beta_3\beta_4\nu^+}{\kappa_3}&0&\frac{\nu^-\beta_1^2-\nu^+\beta_4^2+\kappa_3}{2\kappa_3}&0\\
0&0&0&0
\end{pmatrix}+\mathcal O(|\xi|^{-1}),\end{split}\end{equation}

\begin{equation}\label{3.16}\begin{split}P_2(\xi)=&\begin{pmatrix}
\frac{\nu^-\beta_1^2-\nu^+\beta_4^2+\kappa_3}{2\kappa_3}&
0&\frac{\beta_1\beta_2\nu^-}{\kappa_3}&0\\
0&0&0&0\\
\frac{\beta_3\beta_4\nu^+}{\kappa_3}&0&\frac{\nu^+\beta_4^2-\nu^-\beta_1^2+\kappa_3}{2\kappa_3}&0\\
0&0&0&0
\end{pmatrix}+\mathcal O(|\xi|),\end{split}\end{equation}

\begin{equation}\label{3.17}\begin{split}P_3(\xi)=&\begin{pmatrix}
0&0&0&0\\
0&1&0&0\\
0&0&0&0\\
0&0&0&0\\
\end{pmatrix}+\mathcal O(|\xi|),\end{split}\end{equation}

and
\begin{equation}\label{3.18}\begin{split}P_4(\xi)=&\begin{pmatrix}
0&0&0&0\\
0&0&0&0\\
0&0&0&0\\
0&0&0&1\\
\end{pmatrix}+\mathcal O(|\xi|).\end{split}\end{equation}

With the help of Lemmas \ref{lemma3.1}--\ref{lemma3.3}, we can have the following proposition which is concerned with
long time properties of $L^2$--norm for the
solution.
\begin{Proposition}[$L^2$--theory]\label{Prop3.4} It holds that
\begin{equation}\|\text{e}^{t\mathcal B_1}\mathcal U(0)\|_{L^2}\lesssim \text{e}^{\theta t}\| {\mathcal U}(0)\|_{L^2},\label{2.30}\end{equation}
 for any $t\geq 0$.
\end{Proposition}

\subsection{Spectral analysis for system \eqref{2.2}} We consider the Cauchy problem of \eqref{2.2} with the initial data
 \begin{equation}\label{3.20}(\phi^+, \phi^-)\big|_{t=0}=(\Lambda^{-1}{\rm curl}\tilde{u}^{+}_0, \Lambda^{-1}{\rm curl}\tilde{u}^{-}_0)(x).
 \end{equation}
From the classic theory of the heat equation, it is clear that the solution $\mathcal V=(\phi^+, \phi^-)^t$  to the IVP \eqref{2.2} and \eqref{3.20} satisfies the following decay
estimates.
\begin{Proposition}[$L^2$--theory]\label{Prop3.5} It holds that
\begin{equation}\|\text{e}^{-\nu^\pm t\Lambda^2}\mathcal V(0)\|_{L^2}\lesssim
\| {\mathcal V}(0)\|_{L^2},\nonumber\end{equation}
 for any $t\geq 0$.
\end{Proposition}

We consider the Cauchy problem of \eqref{1.28} with the initial data
\begin{equation}\label{3.21}
\left(\tilde n^{+}, \tilde u^{+}, n^{-},\tilde u^{-}\right)(x, 0)=\left(n_{0}^{+}, \tilde u_{0}^{+}, n_{0}^{-},\tilde u_{0}^{-}\right)(x) \rightarrow(0, \overrightarrow{0}, 0, \overrightarrow{0}), \quad \text { as }|x| \rightarrow+\infty,
\end{equation}
By virtue of the definition of $\varphi^{\pm}$ and $\phi^{\pm}$, and the
fact that the relations
$$\tilde{u}^{\pm}=-\wedge^{-1}\nabla\varphi^{\pm}-\wedge^{-1}\text{div}\phi^{\pm},$$
involve pseudo--differential operators of degree zero, the estimates
in space $H^k(\mathbb R^3)$ for the original function $\tilde{u}^{\pm}$ will
be the same as for $(\varphi^{\pm}, \phi^{\pm})$. Combining Propositions
\ref{Prop3.4} and  \ref{Prop3.5}, we have the
following result concerning long time properties for the solution
semigroup $\text{e}^{t\mathcal{A}}$.
\begin{Proposition}\label{Prop3.6} The global solution $\tilde{U}=(\tilde{n}^+,\tilde{u}^+,\tilde{n}^-,\tilde{u}^-)^t$ of the IVP \eqref{1.28} and \eqref{3.21} satisfies
\begin{equation}\| \text{e}^{t\mathcal{A}}\tilde{U}(0)\|_{L^2}\lesssim \text{e}^{\theta t}
\| \tilde{U}(0)\|_{L^2}.\label{3.22}\end{equation}
\end{Proposition}

\bigskip

\section{Nonlinear instability}\label{1section_appendix-2}

We mention that the local existence of strong solutions to a generic compressible
two--fluid model can be established by  using the standard iteration arguments as in \cite{Wen1} whose
details are omitted. We can arrive at the following conclusion:
\begin{Proposition}\label{Prop4.1} Assume that the notations and hypotheses  in Theorem \ref{3mainth}  are in force. For any given initial data $\left(n_0^+,u_0^+,n_0^-,u_0^-\right)\in H^4(\mathbb R^3)$ satisfying $\inf_{x\in\mathbb R^3}\{n_0^\pm+1\}>0$, there exist a $T>0$ and a unique strong solution $(n^+,u^+,n^-,u^-)\in C^0([0,T];H^4(\mathbb R^3))$ to the Cauchy problem \eqref{1.26}--\eqref{1.27}. Moreover, the strong solution satisfies
\begin{equation}\nonumber
\begin{split}
 \mathcal E(t)
\leq C(T) \mathcal E(0),
\end{split}
\end{equation}
where $\mathcal E(t)=\left\|\left(n^+,u^+,n^-,u^-\right)(t)\right\|_{H^4}$.
\end{Proposition}

\vspace{4mm}
\textbf{\textit{Proof of  Theorem \ref{3mainth}.}}  Now we are in a position to prove Theorem \ref{3mainth} by adopting the basic ideas in \cite{Guo1,Jang,Jiang1,Jiang2,WangT}.
In view of Theorem \ref{2mainth}, we can construct a linear solution $\left(\tilde n^+_\vartheta,\tilde u^+_\vartheta,\tilde n^-_\vartheta,\tilde u^-_\vartheta\right)\in C^0([0,\infty);H^4(\mathbb R^3))$ to the linear system \eqref{1.28}. Moreover, without loss of generality, we suppose that
\begin{equation}\mathcal E\left(\tilde n^+_{0,\vartheta},\tilde u^+_{0,\vartheta},\tilde n^-_{0,\vartheta},\tilde u^-_{0,\vartheta}\right)=\left\|\left(\tilde n^+_{0,\vartheta}, \tilde u^+_{0,\vartheta},\tilde n^-_{0,\vartheta}, \tilde u^-_{0,\vartheta}\right)\right\|_{H^4}=1.\nonumber
\end{equation}
Denote $\left(n^{+,\varepsilon}_{0,\vartheta},u^{+,\varepsilon}_{0,\vartheta},n^{-,\varepsilon}_{0,\vartheta},u^{-,\varepsilon}_{0,\vartheta}\right)\overset{\triangle}=\varepsilon \left(\tilde n^+_{0,\vartheta},\tilde u^+_{0,\vartheta},\tilde n^-_{0,\vartheta},\tilde u^-_{0,\vartheta}\right)$. Then, by virtue of Proposition \ref{Prop4.1}, there is a positive constant $\varepsilon_0$ which may be quite small such that for any $\varepsilon<\varepsilon_0$, there is a unique local strong solution $\left(n^{+,\varepsilon}_\vartheta,u^{+,\varepsilon}_\vartheta,n^{-,\varepsilon}_\vartheta,u^{-,\varepsilon}_\vartheta\right)\in C^0([0,T];H^4(\mathbb R^3))$ to the Cauchy problem \eqref{1.26}--\eqref{1.27}, emanating from the initial data $\left(n^{+,\varepsilon}_{0,\vartheta},u^{+,\varepsilon}_{0,\vartheta},n^{-,\varepsilon}_{0,\vartheta},u^{-,\varepsilon}_{0,\vartheta}\right)$ with $\mathcal E\left(n^{+,\varepsilon}_{0,\vartheta},u^{+,\varepsilon}_{0,\vartheta},n^{-,\varepsilon}_{0,\vartheta},u^{-,\varepsilon}_{0,\vartheta}\right)=\varepsilon$.

We fix $\varepsilon_0>0$ which may be small enough, then for any $\varepsilon\in(0,\varepsilon_0)$. Define
$$T^*=\sup\left\{t\in(0,T^{\max})\big|\sup\limits_{\tau\in[0,t]}\mathcal E\left(\left(n^{+,\varepsilon}_\vartheta,u^{+,\varepsilon}_\vartheta,n^{-,\varepsilon}_\vartheta,u^{-,\varepsilon}_\vartheta\right)(\tau)\right)\le \varepsilon_0\right\}$$
and
$$T^{**}=\sup\left\{t\in(0,T^{\max})\big|\sup\limits_{\tau\in[0,t]}\left\|\left(n^{+,\varepsilon}_\vartheta,u^{+,\varepsilon}_\vartheta,n^{-,\varepsilon}_\vartheta,u^{-,\varepsilon}_\vartheta\right)(\tau)\right\|_{L^2}\le \varepsilon\varepsilon_0^{-\frac{1}{3}}\text{e}^{\theta t}\right\}$$
where $T^{\max}$ denotes the maximal time of existence. Obviously,  $T^*T^{**}>0$, and furthermore,
\begin{equation}\mathcal E\left(\left(n^{+,\varepsilon}_\vartheta,u^{+,\varepsilon}_\vartheta,n^{-,\varepsilon}_\vartheta,u^{-,\varepsilon}_\vartheta\right)(T^*)\right)= \varepsilon_0\quad \text{if} \quad T^*<\infty,\nonumber
\end{equation}
and
\begin{equation}\left\|\left(n^{+,\varepsilon}_\vartheta,u^{+,\varepsilon}_\vartheta,n^{-,\varepsilon}_\vartheta,u^{-,\varepsilon}_\vartheta\right)(T^{**})\right\|_{L^2}= \varepsilon\varepsilon_0^{-\frac{1}{3}}\text{e}^{\theta T^{**}}\quad \text{if} \quad T^{**}<\infty.\label{4.1}
\end{equation}
Assume $T^*=\infty$, otherwise let $T^\varepsilon=T^*$ and $\delta_0=\varepsilon_0$, we can prove Theorem \ref{3mainth} immediately.
Let
\begin{equation}\label{4.2}T^\varepsilon=\frac{1}{\theta}\ln\frac{2\varepsilon_0}{\varepsilon}\left(\text{i.e.,}\  \varepsilon\text{e}^{\theta T^\varepsilon}=2\varepsilon_0\right)\quad \text{and}\quad \vartheta=\frac{1}{T^\varepsilon}.
\end{equation}
Set $\left(n^+_{d},u^+_{d},n^-_{d},u^-_{d}\right)=\left(n^{+,\varepsilon}_{\vartheta},u^{+,\varepsilon}_{\vartheta},n^{-,\varepsilon}_{\vartheta},u^{-,\varepsilon}_{\vartheta}\right)-\varepsilon(\tilde n^+_\vartheta,\tilde u^+_\vartheta,\tilde n^-_\vartheta,\tilde u^-_\vartheta)$. Noticing that $$\left(n^+_{l},u^+_{l},n^-_{l},u^-_{l}\right)=\varepsilon(\tilde n^+_\vartheta,\tilde u^+_\vartheta,\tilde n^-_\vartheta,\tilde u^-_\vartheta)$$ is also a solution to the linear system  \eqref{1.28} with the initial data $\left(n^{+,\varepsilon}_{0,\vartheta},u^{+,\varepsilon}_{0,\vartheta},n^{-,\varepsilon}_{0,\vartheta},u^{-,\varepsilon}_{0,\vartheta}\right)\in H^2(\mathbb R^3)$, it is clear that $\left(n^+_{d},u^+_{d},n^-_{d},u^-_{d}\right)$ is a solution to the system
\begin{equation}\label{4.3}
\left\{\begin{array}{l}
\partial_{t} n^{+}_d+\beta_{1} \operatorname{div} u^{+}_d=\mathbb{F}_{1}, \\
\partial_{t} u^{+}_d+\beta_{1} \nabla n^{+}_d+\beta_{2} \nabla n^{-}_d-v_{1}^{+} \Delta u^{+}_d-v_{2}^{+} \nabla \operatorname{div} u^{+}_d=\mathbb{F}_{2}, \\
\partial_{t} n^{-}_d+\beta_{4} \operatorname{div} u^{-}_d=\mathbb{F}_{3}, \\
\partial_{t} u^{-}_d+\beta_{3} \nabla n^{+}_d+\beta_{4} \nabla n^{-}_d-v_{1}^{-} \Delta u^{-}_d-v_{2}^{-} \nabla \operatorname{div} u^{-}_d=\mathbb{F}_{4},
\end{array}\right.
\end{equation}
subject to the initial condition
\begin{equation}\label{4.4}
\left(n^{+}_d, u^{+}_d, n^{-}_d, u^{-}_d\right)(x, 0)=0,
\end{equation}
where the nonlinear terms are given by
\[
\mathbb{F}_{1}=\alpha_{1} F_{1}\left(\frac{n^{+,\varepsilon}_\vartheta}{\alpha_{1}}, \frac{u^{+,\varepsilon}_\vartheta}{\sqrt{\alpha_{1}}}\right), \quad \mathbb{F}_{2}=\sqrt{\alpha_{1}} F_{2}
\left(\frac{n^{+,\varepsilon}_\vartheta}{\alpha_{1}}, \frac{u^{+,\varepsilon}_\vartheta}{\sqrt{\alpha_{1}}}, \frac{n^{+,\varepsilon}_\vartheta}{\alpha_{4}}, \frac{u^{+,\varepsilon}_\vartheta}{\sqrt{\alpha_{4}}}\right),
\]
and
\[
\mathbb{F}_{3}=\alpha_{4} F_{3}\left(\frac{n^{+,\varepsilon}_\vartheta}{\alpha_{4}}, \frac{u^{+,\varepsilon}_\vartheta}{\sqrt{\alpha_{4}}}\right), \quad \mathbb{F}_{4}=\sqrt{\alpha_{4}} F_{4}\left(\frac{n^{+,\varepsilon}_\vartheta}{\alpha_{1}}, \frac{u^{+,\varepsilon}_\vartheta}{\sqrt{\alpha_{1}}}, \frac{n^{+,\varepsilon}_\vartheta}{\alpha_{4}}, \frac{u^{+,\varepsilon}_\vartheta}{\sqrt{\alpha_{4}}}\right).
\]
\bigskip

Now, we claim that
\begin{equation}\label{4.5} T^\varepsilon=\min\{T^\varepsilon,T^{**}\},
\end{equation}
provided that $\varepsilon_0$ is small enough. Indeed, if $T^{**}=\min\{T^\varepsilon,T^{**}\}$, then $T^{**}<\infty$.
By defining $U=(n^+_d, u^+_d, n^-_d, u^-_d)^t$ and $\mathcal
F=(\mathcal{F}^1,\mathcal{F}^2,\mathcal{F}^3,\mathcal{F}^4)^t$, it holds from Duhamel's principle that
\begin{equation}\nonumber U=\int_0^t\text{e}^{(t-\tau)\mathcal{A}}\mathcal F(\tau)\mathrm{d}\tau.
\end{equation}
By virtue of Proposition \ref{Prop3.6} and \eqref{4.2},  we
have after a complicated but straightforward computation that
\begin{equation}\label{4.6}\begin{split}\|U(T^{**})\|_{L^2}\lesssim& \int_0^{T^{**}}\left\|\text{e}^{(t-\tau)\mathcal{A}}\mathcal F(\tau)\right\|_{L^2}\mathrm{d}\tau\\ \lesssim &\int_0^{T^{**}}\text{e}^{\theta(t-\tau)}\left\|\mathcal F(\tau)\right\|_{L^2}\mathrm{d}\tau
\\ \lesssim &\int_0^{T^{**}}\text{e}^{\theta(t-\tau)}\left(\left\|\left(n^{+,\varepsilon}_{\vartheta},
u^{+,\varepsilon}_{\vartheta},n^{-,\varepsilon}_{\vartheta},u^{-,\varepsilon}_{\vartheta}\right)(\tau)\right\|_{L^2}\left\|\nabla\left(n^{+,\varepsilon}_{\vartheta},
u^{+,\varepsilon}_{\vartheta},n^{-,\varepsilon}_{\vartheta},u^{-,\varepsilon}_{\vartheta}\right)(\tau)\right\|_{W^{1,\infty}}\right.
\\&+\left.\left\|\nabla\left(n^{+,\varepsilon}_{\vartheta},
u^{+,\varepsilon}_{\vartheta},n^{-,\varepsilon}_{\vartheta},u^{-,\varepsilon}_{\vartheta}\right)(\tau)\right\|_{L^4}^2\right)\mathrm{d}\tau
\\ \lesssim &~\int_0^{T^{**}}\text{e}^{\theta(t-\tau)}\varepsilon\varepsilon_0^{-\frac{1}{3}}\text{e}^{\theta \tau}\left(\left(\varepsilon\varepsilon_0^{-\frac{1}{3}}\text{e}^{\theta \tau}\right)^{\frac{1}{6}}\varepsilon_0^{\frac{5}{6}}+\left(\varepsilon\varepsilon_0^{-\frac{1}{3}}\text{e}^{\theta \tau}\right)^{\frac{1}{8}}\varepsilon_0^{\frac{7}{8}}\right)\mathrm{d}\tau
\\ \lesssim &~\varepsilon\varepsilon_0^{-\frac{1}{3}}\text{e}^{\theta T^{**}}\left(\left(\varepsilon\text{e}^{\theta T^{**}}\right)^{\frac{1}{6}}\varepsilon_0^\frac{7}{9}+\left(\varepsilon\text{e}^{\theta T^{**}}\right)^{\frac{1}{8}}\varepsilon_0^\frac{5}{6}\right)\\ \lesssim &~\varepsilon_0^\frac{17}{18}\left(\varepsilon\varepsilon_0^{-\frac{1}{3}}\text{e}^{\theta T^{**}}\right), \end{split}\end{equation}
 where, by H\"older's inequality and Sobolev's inequality, we used the facts
$$\|\nabla f\|_{L^\infty}\lesssim \|f\|_{L^2}^\frac{1}{6}\|\nabla^3 f\|_{L^2}^\frac{5}{6},$$
$$\|\nabla^2 f\|_{L^\infty}\lesssim \|f\|_{L^2}^\frac{1}{8}\|\nabla^4 f\|_{L^2}^\frac{7}{8}$$
and
$$\|\nabla f\|_{L^4}\lesssim \|f\|_{L^2}^\frac{9}{16}\|\nabla^4 f\|_{L^2}^\frac{7}{16}.$$
If $\varepsilon_0$ is small enough, by Proposition \ref{Prop2.2} and \eqref{4.6}, we see that
\begin{equation}\nonumber\left\|\left(n^{+,\varepsilon}_\vartheta,u^{+,\varepsilon}_\vartheta,n^{-,\varepsilon}_\vartheta,u^{-,\varepsilon}_\vartheta\right)(T^{**})\right\|_{L^2}
\le C\left( \varepsilon\text{e}^{\theta T^{**}} +\varepsilon_0^\frac{17}{18}\left(\varepsilon\varepsilon_0^{-\frac{1}{3}}\text{e}^{\theta T^{**}}\right)\right)<\varepsilon\varepsilon_0^{-\frac{1}{3}}\text{e}^{\theta T^{**}},
\end{equation}
which contradicts with \eqref{4.1}.

Finally, performing the similar procedure as in \eqref{4.6} and using Proposition \ref{Prop2.2}, we deduce that
\begin{equation}\begin{split}&\left\|\left(n^{+,\varepsilon}_\vartheta,u^{+,\varepsilon}_\vartheta,n^{-,\varepsilon}_\vartheta,u^{-,\varepsilon}_\vartheta\right)(T^{\varepsilon})\right\|_{L^2}
\\ \ge&~\text{e}^{(\theta-\vartheta)T^\varepsilon}\varepsilon \left\|\left(\tilde n^+_{0,\vartheta},\tilde u^+_{0,\vartheta},\tilde n^-_{0,\vartheta},\tilde u^-_{0,\vartheta}\right)\right\|_{L^2}-C\varepsilon_0^\frac{17}{18}\left(\varepsilon\varepsilon_0^{-\frac{1}{3}}\text{e}^{\theta T^{\varepsilon}}\right)\\ \ge&~\frac{2\varepsilon_0m_0}{\text{e}}-C\varepsilon_0^\frac{29}{18}\\ \ge&~\frac{\varepsilon_0m_0}{\text{e}},
\end{split}\end{equation}
if $\varepsilon_0$ is small enough, where $m_0=\left\|\left(\tilde n^+_{0,\vartheta},\tilde u^+_{0,\vartheta},\tilde n^-_{0,\vartheta},\tilde u^-_{0,\vartheta}\right)\right\|_{L^2}$. This completes the proof of Theorem \ref{3mainth} by defining
$\delta_0=\min\left\{\varepsilon_0,\frac{\varepsilon_0m_0}{\text{e}}\right\}$. \hfill$\Box$

\bigskip

{\bf Statement:}  No conflict of interest exists in the submission of this
manuscript, and 
the datasets generated during and/or analysed during the current study are available from the corresponding author on reasonable request.

\section*{Acknowledgments}
 Yinghui Zhang' research is
partially supported by Guangxi Natural Science Foundation
$\#$2019JJG110003, $\#$2019AC20214 and Key Laboratory of Mathematical and Statistical Model (Guangxi Normal University), Education Department of Guangxi Zhuang Autonomous Region. Lei Yao's research  is
	partially  supported by National Natural Science Foundation of China
	$\#$12171390, $\#$11931013, $\#$11571280 and Natural Science Basic Research Plan
	for Distinguished Young Scholars in Shaanxi Province of China (Grant
	No. 2019JC-26).

\end{document}